\documentclass[a4paper, 11pt]{amsart}
\usepackage{amsfonts, amsmath, amssymb, amsthm,mathtools, url,dsfont}
\usepackage{geometry}
\usepackage{graphicx}
\usepackage[english] {babel}
\usepackage{enumerate}
\usepackage{array}
\usepackage{color}
\usepackage[breaklinks]{hyperref}

\newtheorem{theorem}{Theorem}[section]

\newtheorem{lemma}[theorem]{Lemma}

\renewcommand{\Re}{\operatorname{Re}}
\newcommand{\Spoly}{\mathcal{S}_{\textup{poly}}}
\begin{document}

\begin{abstract}
We consider  $L$-functions $L_1,\ldots,L_k$ from the Selberg class which have polynomial Euler product and satisfy Selberg's orthonormality condition. We show that on every vertical line $s=\sigma+it$ with $\sigma\in(1/2,1)$, these $L$-functions simultaneously take large values of size $\exp\left(c\frac{(\log t)^{1-\sigma}}{\log\log t}\right)$  inside a small neighborhood. 
\end{abstract}

\title[]{Joint extreme values of $L$-functions}

\author{Kamalakshya Mahatab} 
\address{Department of Mathematics and Statistics, University of Helsinki, P. O. Box 68, FIN
00014 Helsinki, Finland}
\email{accessing.infinity@gmail.com, kamalakshya.mahatab@helsinki.fi}

\author{\L ukasz Pa\'nkowski} 
\address{Faculty of Mathematics and Computer Science, Adam Mickiewicz University, Uniwersytetu Pozna\'nskiego 4, 61-614 Pozna\'n, Poland}
\email{lpan@amu.edu.pl}

\author{Akshaa Vatwani} 	
\address{Department of Mathematics,	Indian Institute of Technology Gandhinagar, Palaj, Gandhinagar, Gujarat 382355, India}
\email{akshaa.vatwani@iitgn.ac.in }

\thanks{ KM was supported by Grant 227768 of the Research Council of Norway and Project 1309940 of Finnish Academy. \L P was partially supported by the grant no. 2016/23/D/ST1/01149 from the National Science Centre. Part of this  work was carried out when KM was visiting Adam Mickiewicz University, Pozna\'n with support from a grant of \L P and also when he was a Leibniz fellowship at MFO, Oberwolfach. AV
was supported by  the DST INSPIRE Faculty Award Program and  grant  no. ECR/2018/001566 from  SERB-DST.  KM and AV are deeply  grateful to Adam Mickiewicz University for the kind hospitality during their visit.    }

\maketitle

\section{Introduction}

The problem of finding extreme values of the Riemann zeta-function on a given vertical line lying in the right closed half of the critical strip was first investigated by Titchmarsh in 1927 \cite{titch1927}. However,  the main development of this theory transpired  in the 70's of the twentieth century due to the work of Levinson \cite{lev}, Balasubramanian and Ramachandra \cite{BaluRama}, and  Montgomery \cite{montgomery}. It is worth mentioning that Montgomery introduced a new technique based on Dirichlet's theorem on homogeneous Diophantine approximation to prove that for any fixed $\sigma\in (1/2,1)$, any real number $\theta$ and every sufficiently large $T$,  there exists  $t \in [T^{(\sigma-1/2)/3}, T]$ such that  
\begin{equation}\label{mont}
\Re e^{-i\theta}\log\zeta(\sigma+it) \geq \frac{1}{20}\left(\sigma-\frac{1}{2}\right)^{1/2} \frac{(\log T)^{1-\sigma}}{(\log\log T)^{\sigma}}.
\end{equation}
Moreover, he showed that under the Riemann Hypothesis, the above inequality can be extended to $\sigma\in[1/2,1)$ with a slightly better constant and better range of $t$. Such a result  was proved unconditionally in \cite{BaluRama} for $\sigma=1/2$ and $\theta=0$.

Another important breakthrough was achieved very recently by Bondarenko and Seip in \cite{bs1, bs2}. Based on the resonance method introduced by Soundararajan in \cite{sound} and a connection between extreme values of the Riemann zeta-function and the so-called GCD sums (see \cite{aist,hilb}), they showed that
\[
\max_{0\leq t\leq T}\left|\zeta\left(\frac{1}{2}+it\right)\right|\geq \exp\left(\left(1+o(1)\right)\sqrt{\frac{\log T\log\log\log T}{\log\log T}}\right).
\]
Later the above result was optimized by de~la Bret\`eche and Tenenbaum \cite{dlb}. For $\sigma=1$, it was proved by the first author,  Aistleitner and Munsch \cite{aistmunsch}  that
\[\max_{T^{1/2}\leq t\leq T}\left|\zeta\left(1+it\right)\right|\geq e^\gamma(\log\log T + \log\log\log T +O(1)).\]
In view of the aforementioned results it is natural to ask if similar $\Omega$-results hold for other $L$-functions in number theory. A partial answer to this question was given by Balakrishnan \cite{Bala} for the Dedekind zeta-function as well as  by Sankaranarayanan and Sengupta \cite{SankarSengupta} for a certain class of $L$-functions with real Dirichlet coefficients, by generalizing Montgomery's argument to obtain the analogue of \eqref{mont} for those $L$-functions. It is quite remarkable  that Montgomery's approach as well as the resonance method do not seem  (see \cite{ap}) strong enough to get $\Omega$-results even of the same size as in \eqref{mont} for general $L$-functions, if the assumption that  the Dirichlet coefficients are real is dropped. More precisely, as was shown in \cite{ap}, the best known result for $L$-functions from the Selberg class having polynomial Euler product and satisfying Selberg's normality condition is
\[
\max_{t\in[T,2T]} |L(\sigma+it)|\geq \exp\left( \left(C_L(\sigma)+o(1)\right)\frac{(\log T)^{1-\sigma}}{(\log\log T)^{\theta(\sigma)}}\right),
\]
where $C_L(\sigma)$ is an explicitly given positive constant, $\sigma\in[1/2,1)$ is fixed, $T$ is sufficiently large, and $\theta(1/2)=1/2$, $\theta(\sigma)=1$ otherwise. 

Let us recall that the Selberg class $\mathcal{S}$ consists of meromorphic functions $L(s)$ satisfying the following axioms.
\begin{enumerate}[(i)]
	\item {\bf Dirichlet series}: $L(s)$ can be expressed as a Dirichlet series
	\begin{equation*}
	L(s) = \sum_{n=1}^{\infty} \frac{a_L(n)}{n^s},
	\end{equation*}
	which is absolutely convergent in the region $\Re(s)>1$. 
	
	\item {\bf Analytic continuation}: There exists a non-negative integer $m$, such that $(s-1)^m L(s)$ is an entire function of finite order.
	\item {\bf Functional equation}: 
	$L(s)$ satisfies the functional equation 
	\begin{equation*} \label{fneq2}
	\Phi(s) = \theta \overline{\Phi(1-\overline{s})},
	\end{equation*}
	where 
	\begin{equation*}\label{fneq1}
	\Phi(s) := L(s)Q^s \prod_{j=1}^k \Gamma(\lambda_j s + \mu_j),
	\end{equation*}
	$|\theta|=1$, $Q\in \mathbb{R}$, $\lambda_j \geq 0$, and $\mu_j \in \mathbb{C}$ with $\text{Re } \mu_j\ge 0$. 
	
	\item {\bf Ramanujan hypothesis} - For any $\epsilon >0$, 
	\begin{equation*}\label{rhyp}
	|a_F(n)| = O_\epsilon(n^\epsilon).
	\end{equation*}
	
	\item {\bf Euler product} - There is an Euler product of the form
	\begin{equation*}\label{eprod}
	L(s) = \prod_{p \text{ prime}} L_p(s)
	\end{equation*}
	for $\Re(s)>1$, where
	\begin{equation*}
	L_p(s) =  \exp \left( \sum_{j=1}^\infty \frac{b_L(p^j)}{p^{js}}\right) 
	\end{equation*}
	with $b_L(p^j) \in \mathbb{C}$, $b_L(p^j) = O(p^{j \delta})$ for some $\delta < 1/2$.
	
\end{enumerate}

Let us mention that in some cases it is convenient to assume a stronger axiom than (v), namely the so-called polynomial Euler product:
\begin{enumerate}
	\item[(v')] For $\Re(s)>1$,
	\[L(s)=\prod_p\prod_{j=1}^{r_L}\left(1-\frac{\nu_{L,j}(p)}{p^s}\right)^{-1},\]
	where $\nu_{L,j}(p)$ are complex numbers.
\end{enumerate}
We shall denote by $\Spoly$ the subclass of $\mathcal{S}$ consisting of those $L$-functions that satisfy the axiom (v').
It is worth  emphasizing that in all likelihood this stronger assumption does not exclude any interesting $L$-functions in number theory, since all known examples of elements from the Selberg class also satisfy  the axiom (v'), at least under some widely believed conjectures. From our point of view, the most important consequence of (v') is the fact that in addition with the Ramanujan hypothesis (iv), it implies that $|\nu_{L,j}(p)|\leq 1$ (see \cite[Lemma 2.2]{steuding}). Therefore we have
	\[
a_L(p)=b_L(p) = \sum_{j=1}^{r_L}\nu_{L,j}(p)\qquad\text{and}\qquad |a_L(p)|\leq r_L.
\] 

An interesting phenomenon was observed by Bombieri and Hejhal in \cite{BH}, namely they showed the statistical independence of any collection of $L$-functions under a stronger version of Selberg's orthogonality condition. Let us recall that Selberg's condition claims that
\begin{equation}\label{eq:SelbergA}
\sum_{p\leq x}\frac{|a_L(p)|^2}{p} = \kappa_L\log\log x + O(1)
\end{equation}
and, for any distinct primitive functions $L_1, L_2\in\mathcal{S}$, we have 
\begin{equation}\label{eq:SelbergB}
\sum_{p\leq x}\frac{a_{L_1}(p)\overline{a_{L_2}(p)}}{p} = O(1).
\end{equation}
Very recently it was proved by Lee, Nakamura and the second author \cite{LNP} that under some natural assumptions on $L_1,\ldots,L_n\in\mathcal{S}$ and a stronger version of \eqref{eq:SelbergA} and \eqref{eq:SelbergB} for every pair $L_i,L_j$, $i\ne j$, the functions $L_1,\ldots,L_n\in\mathcal{S}$ are jointly universal in the Voronin sense. Roughly  speaking, this means that any non-vanishing analytic functions $f_1(s),\ldots,f_n(s)$ can be approximated uniformly by certain shifts $L_1(s+i\tau),\ldots,L_n(s+i\tau)$ (see \cite[Theorem 1.2]{LNP}). Thus one can easily deduce that $L$-functions, which pairwise satisfy Selberg's orthogonality condition, are in some sense independent. It is hence natural to ask whether on the vertical segment $[\sigma_0+iT,\sigma_0+2iT]$ with $1/2<\sigma_0<1$,  they can take simultaneously or at least in a small neighborhood, extreme values of the size $\exp\left(c\frac{(\log T)^{1-\sigma_0}}{\log\log T}\right)$ for some positive constant $c$. Using a modification of Montgomery's approach \cite{montgomery} together with an idea of Good  \cite{good}, we prove the following result, answering this question in the affirmative. 

\begin{theorem}\label{thm:main}
	Let $\sigma_0\in(1/2,1)$ be fixed and $L_1(s), \hdots, L_k (s)$ be distinct elements of $\Spoly$, whose Dirichlet coefficients $a_{L_j}$ ($j=1,2,\ldots,k$) satisfy the following strong version of Selberg's orthonormality condition
	\begin{equation}\label{SSOC}
	\sum_{p\le x} a_{L_i}(p)  \overline{a_{L_j}(p)}=  \begin{cases}
	\kappa_j\operatorname{Li}(x) + O\left(\frac{x}{\log^A x}\right) &  \text{ if }i=j,\\ 
	O\left(\frac{x}{\log^A x}\right) &  \text{ if }i \neq j,
	\end{cases}
	\end{equation}
	for suitable $\kappa_j>0$. Here and throughout the paper $A$ denotes an arbitrary large positive constant, not necessarily the same at each appearance. 
	
	Moreover, let us assume that there exists $\delta >0$ such that, for $j=1, \hdots, k$,  we have
	\begin{equation} \label{zero density}
	N_{L_j}(\sigma_0, T):=\sharp\{\rho=\beta+i\gamma: L_j(\rho)=0, \beta\geq \sigma_0, \gamma\in [0,T]\} \ll T^{1-\delta}.
	\end{equation}
	Then for every real $\theta_1,\ldots,\theta_k$ and for sufficiently large $T$, there are $t_1, \hdots, t_k \in [T,2T]$ such that  
	\[ \Re e^{-i\theta_j}\log L_j(\sigma_0+ it_j)\gg \frac{(\log T)^{1-\sigma_0 }}{\log \log T}, \quad j =1, \hdots, k \]
	with $|t_i-t_j|\leq 2(\log T)^{(1+\sigma_0)/2}(\log\log T)^{1/2}$ for all $i,j \in \{1, \hdots, k \}$.
\end{theorem}

As was mentioned earlier, the assumption \eqref{SSOC} implies  Selberg's original orthogonality condition \eqref{eq:SelbergA} and \eqref{eq:SelbergB}. Nevertheless, it is quite likely that \eqref{SSOC} is fulfilled by all $L$-functions. We refer to \cite[Section~4]{LNP} for a detailed discussion of this matter. Here we only mention that the evidence for the truth of this conjecture lies in  the fact that there is a grand hypothesis that each $L$-function from the Selberg class can be defined as a suitable automorphic $L$-function, and so far, all known automorphic $L$-functions satisfying Selberg's conjecture  do fulfill  \eqref{SSOC}.

Let us note that putting $\theta_1=\ldots=\theta_k=0$ in Theorem \ref{thm:main} yields that in every interval $[T,2T]$ there is a short interval containing $t_1,\ldots,t_k$ such that $|L_1(\sigma_0+it_1)|,\ldots,|L_k(\sigma_0+it_k)|$ are at least of size $\exp\left(c\frac{(\log T)^{1-\sigma_0}}{\log\log T}\right)$ for some positive constant $c$. On the other hand, taking $\theta_1=\ldots=\theta_k = \pi$ we have that the absolute values of the given $L$-functions can be very small in a close neighborhood, namely at most of size $\exp\left(-c\frac{(\log T)^{1-\sigma_0}}{\log\log T}\right)$. Interestingly, one can put some $\theta_j$'s as $0$, while others as $\pi$, to show that $L$-functions satisfying \eqref{SSOC} are indeed independent, since in a very close neighborhood some of them might take large values while others take extremely small values. Similarly, one can consider the case when $\theta_j=\pm \frac{\pi}{2}$ to deduce similar statements for  arguments of the given $L$-functions.

We may extend Theorem~\ref{thm:main} to the case $\sigma_0=1$ to show joint large values of $L_1, \hdots, L_k$  of size $O(\log\log T)$ and to $\sigma_0=1/2$ under the generalized Riemann hypothesis to show joint large values of size $\exp \left(c\frac{(\log T)^{1/2 }}{\log \log T}\right)$ for some constant $c>0$. The former  involves recalculating Lemma \ref{lem:denseness} for $\sigma_0=1$ and choosing $M$ and $\rho$ as in Section 5 of \cite{ap}. 

It is also possible to further generalize our theorem to a more generic version which includes Dirichlet polynomials having coefficients as random variables \cite{kamal17, seipsaksman}, and also includes other $L$-functions which are not in the Selberg class. Another possible direction for future research is to investigate when two $L$-functions have a high probability of  taking large values in a small neighbourhood  by combining our techniques with recent developments from \cite{arguin, harper, najnudel}.

\section{Auxiliary lemmas}

The following lemma allows us to estimate a given $L$-function by a suitable Dirichlet polynomial. 

\begin{lemma} \label{lem:shortDir}
Let $L(s) \in \Spoly$ be defined by the Dirichlet series $\sum_{n \ge 1} a_L(n) n ^{-s}$ for $\Re s >1 $.  Suppose that $1/2\leq \sigma_0 \leq 1$, $t_0\geq 15$, $\tau=\tau(t_0)=O(t_0)$, and $L(s)$ does not vanish on $\{a+ib : a\geq \sigma_0, |b-t_0| \le 2\tau\}$. Then for any real $\theta$ and $\rho > 2$ we have 
\begin{align*}
\Re e^{-i\theta}\log L(s_0+it)\geq 
 \frac{1}{2} 
\Re\sum_{|\log (p/\rho)| \le 1} \frac{ a_L(p)} { p^{s_0}} \left(1-  \left| \log \frac{p}{\rho}  \right|  \right) + O\left(1+\rho\frac{\tau+\log t_0}{\tau^2}\right),
\end{align*}
for some $t\in[-\tau,\tau]$ , as $T\rightarrow \infty$. 
\end{lemma}

A version of the above lemma for the Riemann zeta function was used by Montgomery \cite{montgomery} to demonstrate large values. Later, various generalizations of Montgomery's lemma have
appeared in \cite{good}, \cite{lukaszSteu} and \cite{kamal18} to exhibit different properties of $L$-functions. Since our formulation of the lemma is slightly different 
from that in the existing literature,  we briefly reprove the lemma below.
\begin{proof}
From \cite[Lemma 4.1]{lukaszSteu} with $\alpha=1/2$, we obtain for any real $\varkappa$, 
\begin{multline*}
\frac{2}{\pi}\int_{-\tau}^{\tau}\log L(s_0+it) \left(\frac{\sin(t/2)}{t}\right)^2 e^{i\varkappa t}d t  \\
=\sum_{p}\sum_{k\geq 1}\sum_{j=1}^{r_L}\frac{\nu_{L,j}(p)^k}{kp^{ks_0}}\max(0, 1-|\varkappa -k\log p|) +O(e^{|\varkappa|\frac{\tau+\log t_0}{\tau^2}}).
\end{multline*}
Now, we use the last equation for $\varkappa = -\log \rho, 0, \log \rho$, and multiply the resulting equations by $\frac{1}{2}e^{-i\theta}, 1, \frac{1}{2}e^{i\theta}$, respectively. Adding these formulae up, taking into account the  elementary estimates
\[
\int_{-\tau}^{\tau}\left(\frac{\sin t/2}{t}\right)^2 \left(1+\cos(\theta+t\log \rho)\right)dt\leq \frac{\pi}{2},
\]
and 
\begin{align*}
\sum_{p}\sum_{k\geq 2}\sum_{j=1}^{r_L}\frac{\nu_j(p)^k}{kp^{ks_0}}\max(0, 1-|\log \rho -k\log p|)&\ll \sum_{2\leq k\leq \log(e\rho)}\sum_{e^{-1}\rho\leq p^k\leq e\rho}\frac{1}{p^{k\sigma_0}}\\
&\ll \sum_{2\leq k\leq \log(e\rho)} \frac{\rho^{1/k-\sigma_0}}{\log \rho}\ll \rho^{1/2-\sigma_0} \ll 1,
\end{align*}
proves the lemma. 
\end{proof}

\begin{lemma} \label{lem:denseness} 
Let $\sigma_0\in (1/2,1)$ be fixed and $L_1(s), \hdots, L_k (s)$ be distinct elements of $\Spoly$ defined by the Dirichlet series $\sum_{n \ge 1} a_{L_j}(n) n ^{-s}$ for $\Re s >1$, $j=1,2,\ldots,k$, whose coefficients satisfy \eqref{SSOC}.
 Let $\Delta_{\rho}$ denote the set of tuples $z= (z_p)_{e^{-1}\rho \le p \le e\rho} $ with  $z_p \in \mathbb C$, $|z_p|\le 1$, and $w_p=(1-|\log(p/\rho)|)$. Moreover, suppose that $g_j(z)$, $j =1, 2,\ldots, k$ are functions on $\Delta_{\rho}$ defined by 
\[
g_j(z) = \sum_{e^{-1}\rho\le p\le e\rho}  \frac{a_{L_j}(p)w_p}{p^{\sigma_0}} z_p, 
\]
and 
$\xi_1, \hdots, \xi_k$ are arbitrary complex numbers satisfying
\begin{equation*} 
|\xi_j| \leq \frac{c_0}{k} \frac{\min_j\kappa_j}{\max_j r_{L_j}}  \rho^{1-\sigma_0} / \log \rho, \qquad  \qquad (j =1, \hdots, k),
\end{equation*}
where $c_0$ is an arbitrary fixed constant such that $c_0<C_{\sigma_0} :=\big(2\frac{\sinh((1-\sigma_0)/2)}{1-\sigma_0}\big)^2$.
Then the system of equations 
\begin{align}\label{sys}
 g_j(z)= \xi_j, \qquad  \qquad (j =1, \hdots, k)
\end{align}
has a solution $z$ in $\Delta_{\rho}$ for all sufficiently large $\rho$. 
\end{lemma}

\begin{proof}
We will first show that the system \eqref{sys} has a solution in  $\Delta_{\rho}$, if and only if for arbitrary $l_j \in \mathbb{C}$, there exists $z \in \Delta_{\rho}$ such that   
\begin{equation}\label{hb}
\sum_{j=1}^k l_j g_j(z) = \sum_{j=1}^k l_j \xi_j. 
\end{equation}
Trivially, every solution of the system \eqref{sys} is a solution of \eqref{hb}.  In order to prove the other direction, 
we  apply the Hahn-Banach separation theorem to the singleton set $\mathcal A= \{ (\xi_1, \hdots , \xi_k)\} \subseteq \mathbb C^k$  and the set $\mathcal B = \{ (g_1(z), \hdots, g_k(z) ) | z \in \Delta_{\rho} \} $,  each of which  is convex in $\mathbb C^k$.  If the sets $\mathcal A$, $\mathcal B$ are disjoint, then there  exists a continuous linear map $T : \mathbb C^k \rightarrow \mathbb C$ and $a, b \in \mathbb R$ such that 
\begin{equation*}
\Re\big( T(g_1(z), \hdots, g_k(z) ) \big) < a<b< \Re \big(T (\xi_1, \hdots , \xi_k)   \big),\qquad  \text{for all } z\in \Delta_{\rho}. 
\end{equation*}
The existence of a solution of \eqref{hb} for arbitrary complex numbers $l_j$ implies that there cannot exist such a map $T$. Thus, the sets $\mathcal A$ and $\mathcal B$ cannot be disjoint,   giving us a solution to the system \eqref{sys}. 

Thus, it suffices to show that \eqref{hb} has a solution in $\Delta_{\rho}$ for arbitrary $l_j\in\mathbb{C}$. Using the definition of $g_j$, it follows that  every complex number of modulus 
\[
\le \sum_{e^{-1}\rho\le p\le e\rho} w_pp^{-\sigma_0} \bigg|\sum_{j=1}^k l_j a_{L_j}(p)\bigg| 
\]  
can be represented by the left hand side of \eqref{hb} as $z$ runs over all the tuples in $\Delta_{\rho}$. 
Since 
\[
\bigg|\sum_{j=1}^k l_j \xi_j \bigg|  \le \sum_{j=1}^k |l_j| |\xi_j |,
\]
it is enough to show that 
\begin{equation} \label{ineq}
\sum_{j=1}^k |l_j| |\xi_j | \le  \sum_{e^{-1}\rho\le p\le e\rho} w_pp^{-\sigma_0} \bigg|\sum_{j=1}^k l_j a_{L_j}(p)\bigg|.  
\end{equation}

We have 
\begin{align}
 &\sum_{e^{-1}\rho\le p\le e\rho} w_pp^{-\sigma_0} \bigg|\sum_{j=1}^k l_j a_{L_j}(p)\bigg|\sum_{j=1}^k |l_j|
 \label{why}
\\
 &\qquad\qquad\ge \frac{1}{\max_j r_{L_j}}\sum_{e^{-1}\rho\le p\le e\rho} w_pp^{-\sigma_0} \bigg|\sum_{j=1}^k l_j a_{L_j}(p)\bigg|^2 
\nonumber
\\ &\qquad\qquad=  \frac{1}{\max_j r_{L_j}} \sum_{j=1}^k |l_j|^2 \sum_{e^{-1}\rho\le p\le e\rho}
 \frac{| a_{L_j}(p)|^2w_p}{p^{\sigma_0}} + \frac{1}{\max_j r_{L_j}} \sum_{i\ne j} l_i \overline{l_j}   \sum_{e^{-1}\rho\le p\le e\rho}
 \frac{a_{L_i}(p) \overline {a_{L_j}(p)}  }{p^{\sigma_0}}w_p 
 \label{rhs}
\end{align}
By partial summation, one can easily show that \eqref{SSOC} implies that
\[
\sum_{p\leq x}|a_{L_j}(p)|^2\log p = \kappa_j x + O\Big(\frac{x}{\log^A x}\Big).
\]
Therefore, we have
\begin{align*}
\sum_{a\rho\leq p\leq b\rho}\frac{|a_{L_j}(p)|^2}{p^{\sigma_0}} &= \kappa_j\frac{\rho^{1-\sigma_0}}{(1-\sigma_0)\log \rho}(b^{1-\sigma_0}-a^{1-\sigma_0}) \\
&\quad+\kappa_j\frac{\rho^{1-\sigma_0}}{(1-\sigma_0)\log^2\rho}\left(b^{1-\sigma_0}((1-\sigma_0)^{-1}-\log b)-a^{1-\sigma_0}((1-\sigma_0)^{-1}-\log a)\right)\\
&\quad+ O\Big(\frac{\rho^{1-\sigma_0}}{\log^A \rho}\Big),\\
\sum_{a\rho\leq p\leq b\rho}\frac{|a_{L_j}(p)|^2\log p}{p^{\sigma_0}}  &= \kappa_j\frac{\rho^{1-\sigma_0}}{1-\sigma_0}(b^{1-\sigma_0}-a^{1-\sigma_0}) + O\Big(\frac{\rho^{1-\sigma_0}}{\log^A \rho}\Big),
\end{align*}
where $a<b$ are arbitrary given positive constants. Now, by splitting the sum $\sum_{e^{-1}\rho\leq p\leq e\rho}$ into $\sum_{e^{-1}\rho\leq p\leq \rho}$ and $\sum_{\rho\leq p\leq e\rho}$, and  applying the above estimates, one can easily get after short calculations, that
\begin{equation*}\label{first term}
\sum_{e^{-1}\rho\le p\le e\rho}\frac{|a_{L_j}(p)|^2 w_p}{p^{\sigma_0}} = C_{\sigma_0} \kappa_j\frac{\rho^{1-\sigma_0}}{\log\rho}+O\Big(\frac{\rho^{1-\sigma_0}}{\log^A \rho}\Big).
\end{equation*}
Similarly, by \eqref{SSOC} we get for $i\ne j$,
\[
\sum_{e^{-1}\rho\le p\le e\rho}
\frac{a_{L_i}(p) \overline {a_{L_j}(p)}  }{p^{\sigma_0}}w_p = O\Big(\frac{\rho^{1-\sigma_0}}{\log^A \rho}\Big),
\]
so the latter sum on the right hand side of \eqref{rhs} is
\begin{align*}
O\Big(\frac{\rho^{1-\sigma_0}}{\log^A \rho}\Big) 
\sum_{i\ne j} |l_il_j|  = O\Big(\frac{\rho^{1-\sigma_0}}{\log^A \rho}\Big) \sum_{j=1}^k |l_j|^2. 
\end{align*}
Hence we see that for every $\varepsilon>0$ and sufficiently large $\rho$ we have
\begin{align*}
\sum_{e^{-1}\rho\le p\le e\rho} w_p p^{-\sigma} \bigg|\sum_{j=1}^k l_j a_{L_j}(p)\bigg| \sum_{j=1}^k |l_j| 
&\ge (C_{\sigma_0}-\varepsilon) \frac{\min_j \kappa_j}{\max_j r_{L_j}}\frac{\rho^{1-\sigma_0}}{\log\rho}\sum_{j=1}^k|l_j|^2
\\
&\ge \frac{C_{\sigma_0}-\varepsilon}{k} \frac{\min_j \kappa_j}{\max_j r_{L_j}}\frac{\rho^{1-\sigma_0}}{\log\rho}\bigg( \sum_{j=1}^k |l_j| \bigg)^2. 
\end{align*}
This proves \eqref{ineq} as required, since $|\xi_j |\leq \frac{c_{0}}{k} \frac{\min_j \kappa_j}{\max_j r_{L_j}}  \frac{\rho^{1-\sigma}}{\log \rho}$ for all $j$ and $c_0<C_{\sigma_0}$.
\end{proof}

We will use the following lemma due to Chen. Henceforth, $\|\cdot \|$ will denote the distance to the nearest integer. 
\begin{lemma}[Chen, \cite{chen}]  \label{lem:chen}
Let $\lambda_1, \hdots, \lambda_n$ and $\alpha_1, \hdots, \alpha_n$ be real numbers and assume the following property:  
for all integers $u_1, \hdots, u_n$ with $|u_j|\le M$, the assertion 
\[
u_1\lambda_1+ \cdots + u_n\lambda_n = 0
\]
implies that $u_1\alpha_1+ \cdots + u_n\alpha_n$ is an integer. Then for all positive real numbers $\delta_1, \hdots, \delta_n$ and for $T_1<T_2$, we have 
\[
\inf_{t \in [T_1, T_2]} \sum_{j=1}^n \delta_j \| \lambda_j t -\alpha_j\|^2 \le \frac{ \Delta}{4} \sin^2\left(\frac{\pi}{2(M+1)}\right) + \frac{\Delta M^n}{4\pi  (T_2-T_1)\Lambda  }, 
\] 
where 
\[
\Delta = \sum_{j=1}^n \delta_j 
\]
 and 
 \[
 \Lambda = \min\left\lbrace  \left| u_1\lambda_1+ \cdots + u_n\lambda_n  \right|  : u_j \in \mathbb{Z}, |u_j| \le M, \sum_{j=1}^n \lambda_ju_j \ne 0
  \right\rbrace .
 \]
\end{lemma}

\begin{lemma}
	\label{lem:random sums}
Let $\sigma_0\in(1/2,1)$ be fixed and $L_j(s) \in \Spoly$, $j=1,2,\ldots,k$ be defined by the Dirichlet series $\sum_{n \ge 1} a_{L_j}(n)n ^{-s}$ for $\Re s >1 $, whose coefficients satisfy the first equality of \eqref{SSOC}. Moreover, let $\theta_p$, $e^{-1}\rho < p \le e\rho$, be any real numbers and $w_p=1- |\log (p/\rho)|$. Then for every positive constant $c'$ and every interval $T^{(l)} = [T+(l-1)T^\mu,T+ lT^\mu]$ with $0<\mu<1$, $1\leq l\leq T^{1-\mu}$ there is some $t_0\in T^{(l)}$ such that
\[\max_{1\leq j\leq k}
\left| 
\sum_{e^{-1}\rho\le p\le e\rho} \frac{  a_{L_j}(p)}  {p^{\sigma_0+it_0}} w_p - \sum_{e^{-1}\rho\le p\le e\rho} \frac{  a_{L_j}(p)e^{-2\pi i \theta_p} }  {p^{\sigma_0}}w_p
\right| \leq c' \frac{\rho^{1-\sigma_0} } { \log \rho}, \qquad (T\to\infty),
\]
where $\rho = \frac{1}{2Mc}\log T$, $c$ is a constant satisfying $c\mu>2\sinh(1)$, and $M$ is a sufficiently large constant that depends on $c'$.
\end{lemma}
\begin{proof}
Using the Cauchy--Schwartz inequality and the fact that $w_p^2\leq w_p$, we have  
	\begin{align}
&\left| \sum_{e^{-1}\rho\le p\le e\rho} \frac{ a_{L_j}(p) w_p e^{-2\pi i \theta_p} } {p^{\sigma_0}}  - \sum_{e^{-1}\rho\le p\le e\rho} \frac{  a_{L_j}(p)w_p}  {p^{\sigma_0+it_0}}  \right| \nonumber\\
&\qquad\qquad\qquad\qquad\qquad= 
\left| \sum_{e^{-1}\rho\le p\le e\rho} \frac{ a_{L_j}(p) w_p}{p^{\sigma_0+it_0}} 
\left( e^{-2\pi i  \big( \theta_p -\frac{t_0 \log p} {2\pi} \big)  }  -1\right) \right|
\nonumber
\\
\nonumber
&\qquad\qquad\qquad\qquad\qquad\ll  \sum_{e^{-1}\rho\le p\le e\rho} \frac{ |a_{L_j}(p)|w_p}{p^{\sigma_0}} 
\left\|  \theta_p - \frac{t_0 \log p}{2 \pi }\right\| 
\\
\label{expr}
&\qquad\qquad\qquad\qquad\qquad\leq \left(  \sum_{e^{-1}\rho\le p\le e\rho} \frac{ |a_{L_j}(p)|^2 w_p}{p^{\sigma_0}} \right)^{\! \!1/2 } \left(  \sum_{e^{-1}\rho\le p\le e\rho} \frac{1}{p^{\sigma_0}}  \left\|  \theta_p - \frac{t_0 \log p}{2 \pi }\right\|^2 \right)^{\!\! 1/2}
	\end{align}

For the latter sum on the right hand side of \eqref{expr}, we shall apply Lemma \ref{lem:chen} with $M$ being a~sufficiently large constant,   $T_1$  and $T_2$ being the endpoints of the interval $T^{(l)}$, 
\[  \lambda_j = \frac{\log p_j}{2\pi},  \qquad  \alpha_j = \theta_{p_j} , \quad \text{ and } \qquad \delta_j  =  \frac{1}{p_j^{\sigma_0}},    \]
where $p_1, \hdots , p_n$ run over all primes lying in the interval $[e^{-1}\rho, e\rho]$. The first condition of Lemma~\ref{lem:chen} is satisfied due to linear independence of the logarithms of primes.  

In order to estimate $\Lambda$, note that $\left|\sum_{e^{-1}\rho\le p_j\le e\rho}u_j\log p_j\right|$ is minimum when $u_j$'s are chosen to be integers such that  $\left|\sum_{\substack{{e^{-1}\rho\le p_j\le e\rho}\\ u_j\geq 0}}u_j\log p_j\right|$ is close to
$\left|\sum_{\substack{{e^{-1}\rho\le p_j\le e\rho}\\ u_j< 0}}u_j\log p_j\right|$. Without loss of generality we may assume that $\prod_{\substack{{e^{-1}\rho\le p_j\le e\rho}\\ u_j< 0}}p_j^{-u_j}> \prod_{\substack{{e^{-1}\rho\le p_j\le e\rho}\\ u_j\geq 0}}p_j^{u_j}$. Therefore
\[\left|\sum_{{e^{-1}\rho\le p_j\le e\rho}}u_j\log p_j\right|=\left|\log\left(\frac{\prod_{\substack{{e^{-1}\rho\le p_j\le e\rho}\\ u_j< 0}}p_j^{-u_j}} {\prod_{\substack{{e^{-1}\rho\le p_j\le e\rho}\\ u_j\geq 0}}p_j^{ u_j}}\right)\right|
\gg \left|\log\left(1+\frac{1}{\prod_{\substack{{e^{-1}\rho\le p_j\le e\rho}\\ u_j\geq 0}}p_j^{u_j}}\right)\right|.\]
Since 
\[\prod_{\substack{{e^{-1}\rho\le p_j\le e\rho}\\ u_j\geq 0}}p_j^{u_j}\ll \exp\left(M\sum_{e^{-1}\rho\le p\le e\rho}\log p\right)\leq \exp\left(M(2\sinh(1)+o(1))\rho\right),\]
we have the lower bound 
\[
\Lambda \gg \exp(-M(2\sinh(1)+o(1)) \rho ). 
\]

By Lemma \ref{lem:chen}, we have 
\begin{multline}\label{chen ineq}
\inf_{t \in T^{(l)}}\sum_{e^{-1}\rho\le p\le e\rho} \frac{1}{p^{\sigma_0}}  \left\|  \theta_p - \frac{t \log p}{2 \pi }\right\|^2 \\\ll \Delta  \left ( \frac{1}{(M+1)^2} + \frac{M^{(2\sinh(1)+o(1))\rho}   \exp (M(2\sinh(1)+o(1))\rho) }{T^\mu} \right), 
\end{multline}
where $\Delta = \sum_{e^{-1}\rho\le p\le e\rho} p^{-\sigma_0}$. 
Since $\rho = \log T/ (2Mc)$ and $c\mu>2\sinh(1)$,  we see that the term in parenthesis on the right hand side of \eqref{chen ineq} is 
\begin{align*}
&\ll\frac{1}{(M+1)^2} + \exp \left( 2M\rho(2\sinh(1)+o(1)-c\mu)  \right) 
\\
&=\frac{1}{(M+1)^2} + o(1).
\end{align*}
Since $\Delta \ll \frac{\rho^{1-\sigma_0}}{\log \rho}$, we have obtained that there is $t_0\in T^{(l)}$ such that the expression in \eqref{expr} is 
\begin{equation}
\ll \frac{1}{M+1} \left(  \sum_{e^{-1}\rho\le p\le e\rho} \frac{ |a_{L_j}(p)|^2 w_p}{p^{\sigma_0}} \right)^{\! \!1/2 } \left(  \frac{\rho^{1-\sigma_0}} { \log \rho} \right)^{1/2}. 
\end{equation}
Thus,  by taking $M$ sufficiently large  (depending on  $c'$), the proof is complete.
\end{proof} 

\section{Proof of Theorem \ref{thm:main}}
First, let us choose $\mu< \delta$  and define the intervals $T^{(l)} = [T+(l-1)T^{\mu}, T+lT^{\mu} ]$, $1\le l\le T^{1-\mu}$. By \eqref{zero density}, we see that for sufficiently large $T$, there exists an index $l_0 $, with $2\le l_0\le T^{1-\mu}-1$, such that none of the $L_j(s)$'s vanish on 
\[
\big\{a+ib : a\ge \sigma_0, \: b \in T^{(l_0-1)}\cup T^{(l_0)}\cup T^{(l_0+1)} \big\}.
\] 
Put $\tau = (\log T)^{(1+\sigma_0)/2}(\log\log T)^{1/2}$ and $\rho = (\log T)/(2Mc)$, where $c$ is the positive constant defined in Lemma \ref{lem:random sums} and $M$ will be chosen later. Then using Lemma \ref{lem:shortDir} gives for any real numbers $\theta_1, \hdots, \theta_k$, 
\begin{align}\label{short_sum}
\Re e^{-i\theta_j}\log L_j(s_0+it'_j)\geq 
\frac{1}{2} 
\Re\sum_{|\log (p/\rho)| \le 1} \frac{ a_{L_j}(p)} { p^{s_0}} w_p + O\left(\frac{(\log T)^{1-\sigma_0}}{M\log\log T}\right),
\end{align}
for arbitrary given $t_0\in T^{(l_0)}$ and some $t'_1,\ldots,t'_k\in[-\tau,\tau]$, where $w_p = 1-  \left| \log \frac{p}{\rho}  \right|$.

  Let $\theta_p$, $\rho e^{-1} < p \le \rho e$, be  real numbers to be chosen suitably later. By Lemma \ref{lem:random sums}, we obtain that  for some $t_0 \in T^{(l_0)}$, the right hand side of \eqref{short_sum} equals
\begin{equation} \label{theta p}
\frac{1}{2}\Re\sum_{|\log(p/\rho)|   \le 1} \frac{a_{L_j}(p) w_p}{p^{\sigma_0}} e^{ -2\pi i \theta_p } +E(T), 
\end{equation}
for some $t_0\in T^{(l_0)}$, where $E(T)\leq c'\left(\frac{(\log T)^{1-\sigma_0}}{\log\log T}\right)$ for arbitrary given positive constant $c'$.

To complete the proof, we will show that the factors $e^{ -2\pi i \theta_p } $ can be replaced by arbitrary complex numbers of modulus $\le 1$, following which we can apply Lemma~\ref{lem:denseness}. In order to see this, let us note that applying Lemma 6 of \cite{good} to the system 
\[
\sum_{\rho e^{-1} <p \le e\rho}  \frac{a_{L_j}(p) w_p}{p^{\sigma_0}}  z_p = b_{j},   \qquad  1\leq j\leq k
\]
implies that if this system has  a solution in complex numbers $z_p$, $\rho e^{-1} <p \le e\rho$, with $|z_p|\le 1$,  then one can also find a solution $z'_p$, $\rho e^{-1} <p \le e\rho$ and $|z'_p|\leq 1$, of this system for which the strict inequality $|z'_p|<1$ holds for at most $k+1$  primes $p$. This allows us to conclude that for any given complex numbers $z_p$, $\rho e^{-1}\leq p\leq e\rho$, with $|z_p|\leq 1$, there are $z'_p = e^{-2\pi i\theta_p}$, $\rho e^{-1}\leq p\leq e\rho$, with suitable real $\theta_p$ such that
\begin{equation}\label{dyadic}
\sum_{\rho e^{-1} <p \le e\rho} \frac{a_{L_j}(p) w_p}{p^{\sigma_0}} z_p 
=	 \sum_{\rho e^{-1} <p \le e\rho} \frac{a_{L_j}(p) w_p}{p^{\sigma_0}} 
e^{-2\pi i \theta_p} + O\left(\rho^{-{\sigma_0}}\right), 
\end{equation}  
for $\frac{1}{2}< \sigma_0 < 1$.

Thus, we have 
\begin{align*}
\Re e^{-i\theta_j}\log L_j(\sigma_0+it_0+it'_j)
\gg \Re\sum_{|\log(p/\rho)|   \le 1} \frac{a_{L_j}(p) w_p}{p^{\sigma_0}} z_p + E(T) + O\left(\frac{(\log T)^{1-\sigma_0}}{M\log\log T}\right).
\end{align*}
So, applying Lemma \ref{lem:denseness} with  $\xi_1,\ldots,\xi_k$ being positive real numbers of size $\rho^{1-\sigma_0}/\log\rho$ and taking sufficiently large $M$ (notice that $c'$ tends to $0$ as $M$ tends to $\infty$) completes the proof, since $\rho^{1-\sigma_0}/\log\rho\asymp \frac{(\log T)^{1-\sigma_0}}{M^{1-\sigma_0}\log\log T}$.


\end{document}